\documentclass[11pt]{article}
\usepackage{graphicx}
\usepackage{fullpage}
\usepackage{amsmath}
\usepackage{amsfonts}
\usepackage{amssymb}
\usepackage{url}
\usepackage{verbatim}
\usepackage{bbm}
\usepackage{tikz}
\usepackage{paralist}
\usepackage{mathrsfs}
\usepackage{tabularx}
\usepackage{algpseudocode}
\usepackage{algorithm}
\usepackage{float}
\usepackage{amsthm}
\usepackage{amscd}
\usepackage{latexsym}
\usepackage{txfonts}
\usepackage{hyperref}
\usepackage{xcolor}
\usepackage[normalem]{ulem}

\newcommand{\GCCnu}{$\operatorname{GCC}_\nu$}

\newcommand{\nul}{\operatorname{null}}
\def\ord#1{| #1 |} 

\newcommand{\Gc}{G^c}

\newcommand{\R}{\mathbb{R}}

\newcommand{\sRn}{\mathbb{S}_n(\R)}

\newcommand{\G}{\mathcal{G}}
\newcommand{\SG}{\mathcal{S}(G)}

\newcommand{\mr}{\operatorname{mr}}
\newcommand{\mrp}{\operatorname{mr}_+}
\newcommand{\M}{\operatorname{M}}
\newcommand{\Mp}{\operatorname{M}_+}

\newcommand{\mrn}{\operatorname{mr}_\nu}

\newcommand{\dmm}{\lceil \delta\rceil}

\newtheorem{thm}{Theorem}[section]

\newtheorem{prop}[thm]{Proposition}
\newtheorem{cor}[thm]{Corollary}

\theoremstyle{remark}
\newtheorem{remark}[thm]{Remark}
\newtheorem{conj}[thm]{Conjecture}
\newtheorem{ex}[thm]{Example}

\newtheorem{rem}[thm]{Remark}

\newcommand{\bit}{\begin{itemize}}
\newcommand{\eit}{\end{itemize}}
\newcommand{\ben}{\begin{enumerate}}
\newcommand{\een}{\end{enumerate}}
\newcommand{\beq}{\begin{equation}}
\newcommand{\eeq}{\end{equation}}
\newcommand{\bea}{\begin{eqnarray*}}
\newcommand{\eea}{\end{eqnarray*}}
\newcommand{\bpf}{\begin{proof}}
\newcommand{\epf}{\end{proof}\ms}
\newcommand{\ms}{\medskip}

\title{The Weak Version of the Graph Complement Conjecture and Partial Results for the Delta Conjecture}
\author{
Francesco Barioli\thanks{Department of Mathematics, University of Tennessee at Chattanooga, Chattanooga, TN 37403,  USA 
(francesco-barioli@utc.edu).} \and
Shaun M. Fallat \thanks{Department of Mathematics and Statistics,
University of Regina, Regina, SK, Canada 
(sfallat@uregina.ca).}
\and
Himanshu Gupta\thanks{Department of Mathematics and Statistics,
University of Regina, Regina, SK, Canada 
(Himanshu.Gupta@uregina.ca)}
\and
Zhongshan Li\thanks{Department of Mathematics and Statistics, 
Georgia State University, Atlanta, GA 30303, USA 
(zli@gsu.edu)}
}
\date{\today}
\begin{document}

\maketitle
\begin{abstract}
Since the transformative workshop by the American Institute of Mathematics on the minimum rank of a graph, two longstanding open problems have captivated the community interested in the minimum rank of graphs: the graph complement conjecture and the $\delta$-conjecture. In this paper, we use a classical result of Mader (1972) to establish a weak version of the graph complement conjecture for all key minimum rank parameters. In addition, again using the same result of Mader, we present some extremal resolutions of the $\delta$-conjecture. Furthermore, we incorporate the assumption of the $\delta$-conjecture and extensive work on graph degeneracy to improve the bound in the weak version of the graph complement conjecture. We conclude with a list of conjectured bounds on the positive semidefinite variant of the Colin de  Verdi\`ere number.
\end{abstract}

% edit keywords and MSC as needed

\noindent AMS-MSC: 05C50; 15A03; 05C35. 

\section{Introduction}

Throughout this work all matrices discussed are real and symmetric and the set of $n\times n$ real symmetric matrices is denoted by $\sRn$.  For our purposes, a {\em graph} $G=(V,E)$ means a simple undirected graph (no loops, no multiple edges) with a vertex set $V$ and edge set $E$ (consisting of two-element subsets of $V$). The {\em order of $G$}, denoted by $\ord G$, is the cardinality of its vertex set $V$. The {\em size of $G$} is the cardinality of its edge set and is typically denoted by $m(G)$. The {\em complement of a graph $G=(V,E)$} is the graph $G^c$ on the same set of vertices $V$, and with edges $\{i,j\}$, $i\neq j$ whenever $\{i,j\} \not\in E$.

A common practice in combinatorial matrix theory is to associate a collection of matrices to a graph and study various algebraic invariants derived from this matrix class. This approach helps reveal interesting implications about the properties of the graph, or vice versa, uses the structure of a graph to derive algebraic consequences about the matrix class.
One such instance of this theoretical association is the so-called
{\em minimum rank problem for graphs}. In general, the minimum rank problem for a graph $G$ seeks to determine the smallest
possible rank over the collection of all real  symmetric matrices $A=[a_{ij}]$ with the `adjacency property' 
that for all $i \neq j$, $a_{ij} \neq 0$ if and only if $\{i,j\}$ is an edge in $G$. Conversely, 
for $A\in \sRn$,
the {\em graph} of $A$, denoted $\G(A)$, is the graph with vertices
$\{1,\dots,n \}$  and edges $\{ \{i,j \} : ~a_{ij} \ne 0,  1 \le i <j \le n \}$.  

Given a graph $G$, the {\em set of symmetric matrices described by  $G$} is defined as 
\begin{align*}
\SG=\{A\in\sRn : \G(A)=G\}.
\end{align*}
For this particular class of matrices we are interested in studying the {\em minimum rank of
$G$}, which is defined and denoted by
\begin{align*}
\mr(G)=\min \, \{ {\rm rank}\, A : A\in\SG \}.
\end{align*}

It is clear that, since the main diagonal of any matrix in $\SG$ is 
ignored when determining $\G(A)$, the minimum rank satisfies $0 \leq \mr(G) \leq n-1$. Moreover, if $G$ contains at least one edge, then $1 \leq \mr(G) \leq n-1$. A basic, yet instructive, example along these lines is the  path
on $n$ vertices, which is known to be the only graph (on $n$ vertices) with minimum rank equal to $n-1$ (see \cite{FH}).

The minimum rank of graphs has garnered sufficient attention in the literature, especially since the 2006 American Institute of Mathematics (AIM) workshop {\em``Spectra of Families of  Matrices described by Graphs, Digraphs, and Sign Patterns''}. Of the many contributions that resulted from this workshop, several meaningful  open problems have appeared over time, one of which has 
become known as the {\em``Graph Complement Conjecture"} or {\em GCC} for short. The GCC can be stated as the following inequality about
the minimum rank of a graph $G$ and its complement (see also \cite{AIMgroup} and for more information consult \cite{HLS}).

\begin{conj}  [GCC] \label{gcc}
For any graph $G$, 
\begin{align*}
\mr(G) + \mr(G^c) \leq \ord G+2.
\end{align*}
\end{conj}

We acknowledge here that the precise questions posed along these lines at the 2006 AIM workshop were as follows: How large can $\mr(G) + \mr(G^c)$
be? From this, two possibilities may arise:
\begin{enumerate}
    \item[\textbf{Problem 1}] Does there exist a constant $a \geq 2$ such that $\mr(G) + \mr(G^c) \leq \ord G+a$ for all graphs $G$? If so, what is the 
smallest such $a$?
    \item[\textbf{Problem 2}]  Does there exists a constant $b<2$ such that $\mr(G) + \mr(G^c) \leq b\cdot \ord G+2$ for all graphs $G$? If so, what is the smallest such $b$?
\end{enumerate}

The condition on $a$ in problem 1 can easily be seen to satisfy $a \geq 2$. Since the 2006 workshop, it has been known that the best possible value for $a$ is equal to $2$, giving Conjecture~\ref{gcc}. This can be easily seen by considering $G=P_n$ ($n \geq 4$). Since  $\mr(P_n)=n-1$ and $\mr(P_n^c)=3$ (see \cite{AIMgroup}).

 On the other hand, bearing in mind the value of $a=2$ in Problem 1, Problem 2 is very much related to what is now commonly referred to as the {\em Weak version of the Graph Complement Conjecture}. Furthermore, if Conjecture~\ref{gcc} is true in general, then the optimal value of $b$ from Problem 2 is one. Since $\mr(G) \leq \ord G-1$ for any graph, it follows that $b$ in Problem 2 can be chosen to be at most $2$. 

Since the original GCC was recorded, further interesting research has led to several stronger related conjectures, one of which concerns  the minimum rank of all 
positive semidefinite matrices in $\SG$. 
Let $\mrp(G)$ denote the minimum rank of all matrices $A$ in $\SG$ assuming $A$ is
positive semidefinite. Clearly, $\mr(G)\le\mrp(G)$ for any graph $G$. However, a stronger variant of the graph complement conjecture has been proposed
involving $\mrp(G)$. That is,

\begin{conj}   [GCC$_+$] \label{gcc+}
For any graph $G$, 
\begin{align*}
   \mrp(G) + \mrp(G^c) \leq \ord G+2. 
\end{align*}
\end{conj}

It is clear that GCC$_+$ is a stronger inequality than GCC and that the
bound $\ord G+2$ is best possible, considering paths as noted previously. The {\em maximum nullity} of a graph $G$ is defined as
\begin{align*}
\M(G)=\max\, \{ \nul(A): A\in \SG \},    
\end{align*}
and the {\em maximum positive semidefinite nullity of $G$} is defined as
\begin{align*}
  \Mp(G)=\max\, \{\nul (A): A\in \SG,\ A \mbox{ is positive semidefinite}\}.  
\end{align*}
Bearing in mind the elementary rank-nullity theorem, Conjectures \ref{gcc} and \ref{gcc+} are equivalent to 
\begin{align*} 
\M(G) + \M(\Gc)&\geq \ord G-2,\\ 
\M_{+}(G) + \M_{+}(\Gc)&\geq \ord G-2.
\end{align*}

Similar suspected inequalities involving various Colin de Verdi\`ere numbers  have appeared in \cite{KLV}. We first outline the Strong Arnold property for a given matrix.  A real symmetric matrix $A$ satisfies the {\em Strong Arnold property (SAP)} provided there does {not} exist any nonzero real symmetric matrix
$X$ satisfying
$AX = 0$,
 $A\circ X = 0$, and
 $I\circ X=0$,
where $\circ$ denotes the Hadamard (entry-wise) product and $I$ is the identity matrix.
The Strong Arnold property has long been known to be equivalent to requiring that certain manifolds intersect transversally (see, for example, \cite{HLS}). 
%The parameter $\mu(G)$ (\cite{CdVmu}) is the maximum nullity
%among all symmetric matrices $A=[a_{ij}]\in\SG$ that satisfy:
%\bit
%\item $A$ satisfies the strong Arnold property.
%\item For all $i\ne j$,
%$a_{ij}\le 0.$ 
%\item $A$ has exactly one negative eigenvalue %(counting multiplicity).
%\eit
Restricting to positive semidefinite matrices, in  \cite{CdVnu} Colin de Verdi\`ere 
introduced the parameter $\nu(G)$, defined to be the maximum nullity among all positive semidefinite matrices $A\in\SG$ that satisfy the strong Arnold property. 
Clearly, 
for any $G$, $\nu(G)\le \M_+(G) \leq \M(G)$. 
So, it is natural to form a graph complement-type conjecture for $\nu$ as follows.
\begin{conj}   [\GCCnu ]
For any graph $G$, 
\begin{align*} 
\nu(G) + \nu(G^c) \geq \ord G-2.
\end{align*}
\end{conj}

Another related parameter, $\xi(G)$ has also been defined as the maximum nullity over all matrices in $S(G)$ that satisfy $SAP$ (see \cite{BFH-xi}). Evidently, it follows that $\nu(G) \leq \xi(G) \leq M(G)$.  Certainly, in general, \GCCnu\ is stronger than GCC$_{+}$ (and stronger than the corresponding GCC$_{\xi}$, which we do not define here but can be found in \cite{HLS}), that is also stronger than GCC. Note that the GCC, GCC$_+$, and GCC$_{\nu}$ all fall into the family of so-called Nordhaus-Gaddum type problems
(see \cite{AAC}) as they involve bounding the sum of a graph parameter at a graph $G$ and at its complement. Nordhaus–Gaddum type problems have appeared in numerous works and involve a variety of graph parameters, including the chromatic number, independence number, domination number, and others such as the Hadwiger number (see \cite{Kost}). For inequalities specific to certain nullity and rank parameters, see \cite{ng_paper, hogben}.

A core and important property of many  Colin de Verdi\`ere-type parameters is minor monotonicity.  The {\em contraction} of an edge $e=\{u,v\}$ of  $G$ is obtained by identifying the vertices $u$ and $v$, removing a loop   that may arise in this process, and replacing any multiple edges by a single edge.
 A {\em minor}
of $G$ arises by performing a sequence of edge deletions, deletions of isolated
vertices, and/or contractions of edges.  
 A graph parameter $\beta$ is {\em minor monotone} if for any minor $H$  of $G$, $\beta(H) \le \beta(G)$ and
 $\beta(G)=\beta(H)$ if $G$ is isomorphic to $H$.  In \cite{CdVmu} and \cite{CdVnu} it is shown that  $\nu$ is a minor monotone parameter. If $\beta$ is a graph parameter, not necessarily minor monotone, we define the minor-monotone ceiling 
 of $\beta$ as follows:
 \begin{align*}
    \lceil \beta\rceil(G)  &:= \max\, \{ \, \beta(H): H~ \mbox{is a minor of}~ G\}.   
 \end{align*}

To further simplify the notation, we set $\mrn(G)=\ord G-\nu(G)$ for any graph $G$. In this work, we are primarily interested in the so-called {\em Weak Graph Complement Conjectures}, which we formally state now for all relevant parameters.

\begin{conj} [Weak Graph Complement Conjectures] \label{gccnu}
There exist universal constants $b,b_+,b_{\nu} < 2$ such that for every graph $G$, the following inequalities hold:
\begin{enumerate}[(i)]
    \item  $\mr(G) + \mr(G^c) \leq b\cdot\ord G +2$,
    \item  $\mrp(G) + \mrp(G^c) \leq b_+\cdot \ord G+2$, 
    \item $\mrn(G) + \mrn(G^c) \leq b_{\nu}\cdot \ord G+2.$
\end{enumerate}
\end{conj}

The version stated above of the Weak Graph Complement Conjectures represents a slight modification from what has generally been referred to as the ``weak" version of the Graph Complement Conjecture in the community (see Question 2 in \cite{gcc1}). In particular, we have added a `+2' in part to align with the presumed constant from the Graph Complement Conjecture. However, a by-product of the proofs below that provide a positive resolution to Conjecture \ref{gccnu} also proves that there exists a constant $b_\nu<2$ such that $\mrn(G) + \mrn(G^c) \leq b_{\nu}\cdot \ord G$ for all graphs $G$ (see also  Question 2 in \cite{gcc1}).

If Conjecture \ref{gccnu} holds, a natural subsequent problem is to determine the smallest possible values of such constants. By definitions, in Conjecture \ref{gccnu} we have $(iii)\implies (ii) \implies (i)$. In Section \ref{sec_Mader_Consequences}, we note a result of Mader and use it to resolve Conjecture \ref{gccnu} with constants of at most $1.708$ (see Theorem \ref{wggc-nu}). 

As previously mentioned, the 2006 AIM workshop was, in many ways, a springboard for numerous investigations into the minimum rank and the maximum nullity of graphs. Another significant problem posed at this workshop has since become known as the   {\em``$\delta$-conjecture"}. In what follows, we provide further details on the $\delta$-conjecture as well as its strong version. 

The {\em degree} of a vertex $v$ in $G$, denoted by $\deg_G(v)$, is defined as the number of neighbors of $v$ in $G$, while the {\em minimum degree} $\delta(G)$ represents the smallest degree among all vertices of $G$. It is easily seen that the parameter $\delta$ is not minor monotone. This fact suggests the introduction of two additional parameters, namely, the {\em degeneracy} of $G$, defined as
\begin{align*}
    l(G) = \max\, \{ \delta(H) : H \mbox{ is a subgraph of } G \},
\end{align*}
and the {\em minor monotone ceiling} of $\delta(G)$, is given by 
\begin{align*}
    \dmm(G) = \max\, \{\delta(H) : H \mbox{ is a minor of } G \}.
\end{align*}

The degeneracy of a graph is a classical topic in graph theory and can be seen as a measure of the sparsity of a graph. In \cite{LickWhite}, a graph was said to be \emph{$k$-degenerate} if for every induced subgraph $H$ of $G$, $\delta(H)\leq k$. By definition, $l(G)$ is monotone under edge and vertex deletions, while $\dmm(G)$ is minor monotone. Since a subgraph can always be regarded as a minor, we have $\delta(G) \leq l(G) \leq \dmm(G)$. 

The $\delta$-conjecture asks whether $M(G)\geq \delta(G)$, or equivalently $\mr(G) \leq \ord G - \delta(G)$ holds for all graphs $G$. After the original $\delta$-conjecture was conceived and observed, additional analysis has led to stronger conjectures about the minimum rank restricting to certain subclasses of positive semidefinite matrices in $\SG$. The following strong version of this conjecture was proposed in  \cite{barioli2013parameters}.
\begin{conj} [$\delta$-Conjecture for $\nu$] \label{delta}
For any graph $G$, 
\begin{equation} \label{deltaEq}
    \delta(G) \leq l(G) \leq \dmm(G) \leq \nu(G).
\end{equation}
\end{conj}
Note that it is enough to prove $\nu(G) \geq \delta(G)$ for Conjecture \ref{delta} since  $\nu(G)$ is a minor monotone parameter. A classical result in graph theory is the relationship $\delta(G) \geq \kappa(G)$ for any graph $G$, where $\kappa(G)$ is the {\em vertex connectivity} of $G$. The minor monotone ceiling of $\kappa(G)$ is denoted by $\lceil \kappa \rceil (G)$, and therefore, $\lceil \delta \rceil (G)\geq \lceil \kappa \rceil (G)$ for all graphs $G$. The following theorem relates $\nu(G)$ to $\lceil \kappa \rceil (G)$, based on results of Lov\'{a}sz, Saks, Schrijver \cite{LSS1989}, and Holst \cite{vdH08}. Specifically, in \cite{LSS1989}, it was shown that $\Mp(G) \geq \kappa(G)$ for any graph $G$, by deriving a faithful orthogonal representation for a graph $G$ in a generic manner. Building on this, Holst \cite{vdH08} demonstrated that there exists a positive semidefinite matrix satisfying the strong Arnold property, thereby establishing the inequality $\nu(G) \geq \kappa(G)$ for any graph $G$ (see \cite[Theorem 4]{vdH08}). 
Since $\nu$ is a minor monotone parameter, the following result then follows directly.
\begin{thm}\label{thm_LSS_vertex_connectivity}
    Let $G$ be a given graph. Then $\nu(G) \geq \lceil \kappa \rceil (G)$.
\end{thm}

We note here that in fact Maehara's conjecture, as referenced much earlier in \cite{LSS1989}, is equivalent to the $\delta$-conjecture stated for $\mrp(G)$, namely is it true for any graph $G$, $\mrp(G) \leq \ord G-\delta(G)$, or equivalently $M_+(G)\geq \delta(G)$? The progress in this work, refer to Section \ref{subsec_delta}, resolves Conjecture \ref{delta} for several classes of graphs with specific restrictions such as the girth, the minimum degree, and under the constraint of certain classes of forbidden subgraphs (see Theorem \ref{thm_partial_proof_of_delta} for more details). 

In Section \ref{degen-sec}, we incorporate the concept of graph degeneracy along with connections to the $\delta$-conjecture to provide stronger bounds for all three inequalities in Conjecture \ref{gccnu} (see Section \ref{subsec_deg_WGCC}). Finally, we conclude the paper in Section~\ref{sec_other_lower_bounds} with additional queries regarding potentially interesting lower bounds on the parameter $\nu
(G)$, incorporating other well-known graph parameters.

\section{Weak graph complement and delta conjectures: consequences of Mader's theorem}\label{sec_Mader_Consequences}
In this section, we appeal to a famous result of Mader to support progress on the Weak Graph Complement Conjectures by aiming to establish a universal constant $b_\nu$ in Conjecture \ref{gccnu}. Additionally, we use this result to affirm the $\delta$-conjecture~\ref{delta} under certain constraints, specifically in relation to the girth of a graph and the exclusion of particular forbidden subgraphs. Recall that the girth of a graph is the length of a shortest cycle contained in the graph. 

Let $G$ be a graph of order $n$ and size $m$, and let the average degree of $G$ be denoted by $d(G) := 2m/n$. We denote its minor monotone ceiling by $\lceil d \rceil(G)$. The following classical result of Mader \cite{Mader1972} establishes that large average degree guarantees the existence of a subgraph with high connectivity. 
\begin{thm}[Mader \cite{Mader1972}]\label{thm_Mader1972}
     Let $G$ be a graph of order $n$ and size $m$. If $n\geq 2k-1$ and $m\geq (2k-3)(n-k+1)+1$, then $G$ contains a $k$-connected subgraph. In particular, the same conclusion holds if $d(G)\geq 4(k-1)$.
\end{thm}

For a graph $G$, the following theorem establishes a lower bound for $\nu(G)$ based on the number of edges in its complement. 
\begin{thm}\label{thm_lower_bound_size_of_complement}
    Let $G$ be a graph of order $n$ and size $m$, and let $m(G^c)$ be the size of $G^c$. 
    Then 
    $$\nu(G)> \frac{n-1}{2}-\sqrt{\frac{m(G^c)}{2}}.$$
\end{thm}
\begin{proof}
    Using Theorem \ref{thm_Mader1972} and the elementary quadratic formula, $G$ contains a $k$-connected subgraph where $k$ is the greatest integer less than or equal to $\frac{n}{2}+\frac{5}{4}-\sqrt{\frac{n^2-n-2m}{4}+\frac{9}{16}}$. Thus, by using Theorem \ref{thm_LSS_vertex_connectivity} we obtain
    \begin{align*}
    \nu(G)&\geq \lceil \kappa \rceil(G) 
    > \frac{n}{2}+\frac{5}{4}-\sqrt{\frac{n^2-n-2m}{4}+\frac{9}{16}}-1 \geq \frac{n}{2}+\frac{5}{4}-\sqrt{\frac{n^2-n-2m}{4}}-\sqrt{\frac{9}{16}}-1\\
    &= \frac{n-1}{2}-\sqrt{\frac{n^2-n-2m}{4}}
    = \frac{n-1}{2}-\sqrt{\frac{m(G^c)}{2}}. \qedhere
    \end{align*}
\end{proof}

\subsection{Resolution of the Weak Graph Complement Conjectures}
We begin this subsection by resolving the three Weak Graph Complement Conjectures by demonstrating that $b,\ b_+$, and $b_\nu$ in Conjecture \ref{gccnu} can be chosen to be close to $1+\frac{1}{\sqrt{2}} < 1.708$ for large $n$. 

\begin{thm} \label{wggc-nu} 
Let $G$ be a graph of order $n \geq 4$. Then we have 
    \[\mrn(G) + \mrn(G^c) < \left(1+\frac{1}{\sqrt{2}}\right)n+1.\]
\end{thm}
\begin{proof}
   Let $m(G)$ and $m(G^c)$ be the sizes of $G$ and $G^c$, respectively. By using Theorem \ref{thm_lower_bound_size_of_complement}, we have that $\nu(G) > \frac{n-1}{2}-\sqrt{\frac{m(G^c)}{2}}$ and $\nu(G^c) > \frac{n-1}{2}-\sqrt{\frac{m(G)}{2}}$. Therefore,
   \begin{align*}
       \mrn(G) + \mrn(G^c) = 2n - (\nu(G)+\nu(G^c)) < n+1+\left(\frac{\sqrt{m(G)}+\sqrt{m(G^c)}}{\sqrt{2}}\right).
   \end{align*}
   Since $m(G) + m(G^c) = \frac{n(n-1)}{2}$ we obtain
   \begin{align*}
       \mrn(G) + \mrn(G^c) 
       &< n+1+\left(\frac{\sqrt{\frac{n(n-1)}{4}}+\sqrt{\frac{n(n-1)}{4}}}{\sqrt{2}}\right) < n+1+\left(\frac{\sqrt{\frac{n^2}{4}}+\sqrt{\frac{n^2}{4}}}{\sqrt{2}}\right)
       = \left(1+\frac{1}{\sqrt{2}}\right)n + 1. \qedhere
   \end{align*}
\end{proof}

\begin{remark}
    In summary, to prove the Weak Graph Complement Conjecture, we used Mader's Theorem \ref{thm_Mader1972} to show that the graph $G$ and its complement $G^c$ must each contain subgraphs $H_1$ and $H_2$, respectively, such that $\kappa(H_1)+\kappa(H_2) > \left(1-\frac{1}{\sqrt{2}}\right)n$. However, this strategy of identifying highly connected subgraphs in both a graph $G$ and its complement $G^c$ cannot improve the bound in Conjecture \ref{gccnu} beyond $\frac{3}{2}\ord G$. This limitation is illustrated by the following graph family, originally presented in \cite{Matula1983}. 
    
    Start with the path graph $P_4$, and let $s\geq 1$ be an integer. Replace each of the two end vertices of $P_4$ with a copy of the complete graph $K_s$, and each of the two internal vertices with a copy of $K_s^c$. Each edge in $P_4$ is then replaced with all possible edges between the corresponding sets of vertices in the replacements, effectively forming a complete bipartite connection. The resulting graph has $4s$ vertices, is self-complementary, and the maximum vertex-connectivity of a subgraph in it is $s$. This construction was given by Matula in \cite{Matula1983} in the context of a Ramsey-type problem for graph connectivity, serving as a lower-bound example.

    However, we believe that it is worth noting a few possible improvements to the bounds of Theorem \ref{wggc-nu} under certain assumptions on a graph $G$. First, if we assume that the graph $G$ is sufficiently sparse (for example, $m(G) \leq \frac{1}{100}n^2$), then the bound of Theorem \ref{wggc-nu} can be improved from $1.708$ to $1.566$ using the same proof as Theorem \ref{wggc-nu}. Second, one could employ better but more restrictive bounds similar to that of Mader's Theorem \ref{thm_Mader1972} (see \cite{Yuster2003, Bernshteyn2016}) to obtain a slight improvement to Theorem \ref{wggc-nu} under the assumption that the graph is sufficiently sparse. However, we opt not to pursue this direction.
\end{remark}

\subsection{Constrained resolution of the delta conjecture}\label{subsec_delta}
We begin with a result that $\nu(G)$ is at least one-quarter of the value predicted by the $\delta$-conjecture~\ref{delta}.
\begin{thm}\label{thm_nu_is_one_fourth_of_delta}
    Let $G$ be a given graph. Then $\nu(G) > \frac{\lceil d \rceil(G)}{4}$. In particular, $\nu(G) > \frac{\lceil \delta \rceil (G)}{4}$. 
\end{thm}
\begin{proof} Consider a minor $G'$ of $G$ such that $d(G')= \lceil d \rceil(G)$. Mader's Theorem \ref{thm_Mader1972} implies that there exists a subgraph $G''$  of $G'$ (in turn, a minor of $G$) such that $\kappa(G'') > \frac{d(G')}{4} = \frac{\lceil d \rceil(G)}{4}$. Therefore, the first assertion follows from Theorem \ref{thm_LSS_vertex_connectivity}. The second assertion follows immediately since $\lceil d \rceil (G)\geq \lceil \delta \rceil (G)$.
\end{proof}

Thus, if one can show that a graph $G$, under certain assumptions, satisfies $\lceil \delta \rceil(G) \geq 4 \delta (G)$, then $G$ would satisfy the $\delta$-Conjecture~\ref{delta}. In this direction, we present several results illustrating that increasing the girth of a graph can enforce the existence of a minor whose minimum degree is significantly larger than that of the original graph.

\begin{prop}\label{prop_Mader1998}
    Let $G$ be a graph with minimum degree $\delta$ and girth $g$, and let $k\geq 1$ be an integer. If $\delta \geq 3$ and $g\geq 8k+3$, then $\lceil \delta \rceil(G)\geq \delta (\delta-1)^k$.
\end{prop}

See \cite[Lemma 7.2.5]{Diestel} or \cite[Proposition 6]{Kuhn_Osthus_2003} for a proof of Proposition \ref{prop_Mader1998}. It is noted in \cite{Kuhn_Osthus_2003}, before the proof, that ``\emph{its proof is the same as the beginning of Mader’s proof of his main result of \cite{Mader1998}}''. K{\"u}hn and Osthus \cite{Kuhn_Osthus_2003} used a probabilistic version of Mader's argument to improve and generalize Proposition \ref{prop_Mader1998} as follows.

\begin{thm}[K{\"u}hn and Osthus \cite{Kuhn_Osthus_2003}]\label{thm_Kuhn_Osthus_2003}
    Let $G$ be a graph with minimum degree $\delta$ and girth $g$, and let $k\geq 1$ be an integer. Then:
    \begin{enumerate}[(a)]
        \item If $\delta\geq 3$ and $g \geq 4k+3$, then $\lceil \delta \rceil (G) \geq \frac{(\delta-1)^{k+1}}{48}$.
        \item If $\delta \geq \max\{20k,8\cdot 10^6\}$ and $g \geq 4k+1$, then $\lceil \delta \rceil (G) \geq \frac{\delta^{k+1/2}}{288}$.
    \end{enumerate}
\end{thm}

As an immediate consequence of Theorem \ref{thm_nu_is_one_fourth_of_delta} and Theorem \ref{thm_Kuhn_Osthus_2003}, we obtain the following result, which implies that if the girth is sufficiently large, then $\nu(G)$ is superlinear in the minimum degree of  $G$. 
\begin{cor}\label{cor_bound_on_nu_via_girth}
    Let $G$ be a graph with minimum degree $\delta$ and girth $g$ and let $k\geq 1$ be an integer. Then:
    \begin{enumerate}[(a)]
        \item If $\delta\geq 3$ and $g \geq 4k+3$, then $\nu(G) > \frac{(\delta-1)^{k+1}}{192}$.
        \item If $\delta \geq \max\{20k,8\cdot 10^6\}$ and $g \geq 4k+1$, then $\nu(G) > \frac{\delta^{k+1/2}}{1152}$.
    \end{enumerate} 
\end{cor}

Observe that a girth of at least $5$ is the minimum required for the applicability of the above result. Krivelevich and Sudakov \cite{krivelevich2009minors}  studied the existence of large minors in expanding graphs. Among other results, they also proved that if $G$ does not contain an even cycle or a complete bipartite graph, then $G$ contains a minor whose average degree is superlinear in the average degree of $G$, up to the product of a constant factor. Those results are as follows.  

\begin{thm}[Krivelevich and Sudakov \cite{krivelevich2009minors}]\label{thm_Krivelevich_and_Sudakov_2009}
    Let $G$ be a graph with average degree $d$. Then:
    \begin{enumerate}[(a)]
        \item If $G$ is $K_{s,s'}$-free, then $\lceil d \rceil(G) \geq c\cdot d^{1+\frac{1}{2(s-1)}}$, where $2\leq s\leq s'$ and $c = c(s,s')$ is some positive constant.
        \item If $G$ is $C_{2t}$-free, then $\lceil d \rceil(G) \geq c\cdot d^{\frac{t+1}{2}}$, where $t\geq 2$ and $c = c(t)$ is some positive constant.
    \end{enumerate}
\end{thm}

As an immediate consequence of Theorem \ref{thm_nu_is_one_fourth_of_delta} and Theorem \ref{thm_Krivelevich_and_Sudakov_2009} we obtain the following result. 
\begin{cor}\label{cor_bound_on_nu_via_H_free}
    Let $G$ be a graph with average degree $d$. Then:
    \begin{enumerate}[(a)]
        \item If $G$ is $K_{s,s'}$-free, then $\nu(G) > \frac{c}{4}\cdot d^{1+\frac{1}{2(s-1)}}$, where $2\leq s\leq s'$ and $c = c(s,s')$ is some positive constant.
        \item If $G$ is $C_{2t}$-free, then $\nu(G) > \frac{c}{4}\cdot d^{\frac{t+1}{2}}$, where $t\geq 2$ and $c = c(t)$ is some positive constant.
    \end{enumerate}
\end{cor}

Finally, we are in a position to present our main contribution as a partial advancement toward $\delta$-Conjecture~\ref{delta}.
\begin{thm}\label{thm_partial_proof_of_delta}
    Let $G$ be a graph with minimum degree $\delta$ and girth $g$. Then $\nu(G) \geq \delta$, provided that one of the following conditions holds:
    \begin{enumerate}[(a)]
        \item $g\geq 11$ and $\delta\geq 4$;
        \item $g\in \{7,8,9,10\}$ and $\delta \geq 193$;
        \item $g \in \{5,6\}$ and $\delta \geq 8 \cdot 10^6$;
        \item $G$ is $K_{s,s'}$-free and $\delta\geq r$, where $2\leq s\leq s'$ and $r = r(s,s')$ is some positive constant;
        \item $G$ is $C_{2t}$-free and $\delta\geq r$, where $t\geq 2$ and $r = r(t)$ is some positive constant.
    \end{enumerate}
\end{thm}
\begin{proof}
Theorem \ref{thm_nu_is_one_fourth_of_delta} and Proposition \ref{prop_Mader1998} implies that $\nu(G) > \frac{\delta(G)(\delta(G)-1)}{4}$, and hence part~(a) follows. Corollary \ref{cor_bound_on_nu_via_girth}~(a) implies that $\nu(G) > \frac{(\delta-1)^2}{192}$, and hence part~(b) follows. Similarly, part~(c) implied by Corollary \ref{cor_bound_on_nu_via_girth}~(b). The last two parts follow immediately from Corollary \ref{cor_bound_on_nu_via_H_free}, as both bounds are superlinear in $d$, and hence in $\delta(G)$. This completes the proof.
\end{proof}

\begin{rem}
    In summary, to prove $\delta$-Conjecture~\ref{delta}, we showed that under a certain assumption on a graph $G$, $\lceil \kappa \rceil (G)\geq \delta(G)$. However, it is worth noting that this inequality does not hold for all graphs. For example, the Mader graph $M_{12}$ from \cite{Mader1968} has minimum degree $\delta = 5$, while $\lceil \kappa\rceil = 4$ (see \cite[Example A.16]{barioli2013parameters} for details).

    On the other hand, Theorem \ref{thm_partial_proof_of_delta} implies that if the girth of a graph $G$ is at least $5$ and the minimum degree is sufficiently large, then $\lceil \kappa \rceil (G) \geq \delta(G)$. Moreover, even when the girth is $3$ or $4$, the same inequality holds, provided that $G$ forbids an even cycle or a complete bipartite subgraph and its minimum degree is large enough depending on that forbidden subgraph. These observations lead to a natural question, which may be of independent interest: Does there exist a universal constant $c$ such that $\lceil \kappa \rceil(G) \geq \delta(G)$ for all graphs $G$ with $\delta(G)\geq c$?
\end{rem}
\section{Degeneracy of graphs}\label{degen-sec}

In this section, we study the degeneracy of a graph through the lens of the Weak Graph Complement Conjecture. To this end, we explore the properties of graph degeneracy with particular emphasis on investigating the sum $l(G) + l(\Gc)$. 

For any graph $G$, it is easy to verify that $\delta(G)+\delta(\Gc) \leq \ord G-1$, where the equality is only valid for regular graphs. Interestingly, even if $l(G)$ may be substantially larger than $\delta(G)$, a similar inequality holds for $l(G)$, as stated implicitly and proved in \cite{ChartrandMitchem}. We provide a short proof here for completeness.

\begin{prop} \label{lGprop}
    For any graph $G$
    \begin{equation} \label{lG+lGc}
        l(G) + l(\Gc) \leq \ord G-1.
    \end{equation}
    In particular, 
    \begin{equation} \label{min_lG}
        \min\{ l(G), l(\Gc)\} \leq \dfrac{\ord G-1}{2}.
    \end{equation}
\end{prop}

\begin{proof}
    Assume $l(G)+l(\Gc) \geq \ord G$ for some graph $G$. We can then find the subgraphs $H$ of $G$ and $K$ of $\Gc$, such that $\delta(H) = l(G)$ and $\delta(K) = l(\Gc)$. Thus, $|H|+|K| > l(G) + l(\Gc) \geq \ord G$, which implies $V(H) \cap V(K) \neq \varnothing$. Let $v \in V(H) \cap V(K)$. We then have
    \begin{align*}
        \ord G-1 =  \deg_G(v)+\deg_{\Gc}(v)
         \geq  \deg_H(v)+\deg_K(v) 
         \geq  \delta(H)+\delta(K)
         =  l(G)+l(\Gc) 
         \geq  \ord G
    \end{align*}
    namely, a contradiction. Equation (\ref{min_lG}) easily follows from (\ref{lG+lGc}).
\end{proof}

We also observe that (\ref{lG+lGc}) is no longer valid if $l(G)$ is replaced by $\dmm(G)$. In fact, for the graph $G$ in Figure~\ref{counterDmm}, we have $\ord G=8$, while $\dmm(G)=\dmm(\Gc)=4$.

\begin{figure}[H] 
    \begin{center}
        \setlength{\unitlength}{1mm}
        \begin{picture}(40,20)(0,4)
            \put(5,0){\circle*{1.5}}
            \put(5,20){\circle*{1.5}}
            \put(10,5){\circle*{1.5}}
            \put(10,15){\circle*{1.5}}
            \put(20,5){\circle*{1.5}}
            \put(20,15){\circle*{1.5}}
            \put(30,5){\circle*{1.5}}
            \put(30,15){\circle*{1.5}}
            \qbezier(5,0)(5,0)(5,20)
            \qbezier(5,0,)(5,0)(10,5)
            \qbezier(5,20)(5,20)(10,15)
            \qbezier(5,0)(12,1)(20,5)
            \qbezier(5,20)(12,19)(20,15)
            \qbezier(5,0)(15,-2)(30,5)
            \qbezier(5,20)(15,22)(30,15)
            \qbezier(10,5)(10,5)(10,15)
            \qbezier(20,5)(20,5)(20,15)
            \qbezier(30,5)(30,5)(30,15)
            \qbezier(10,5)(20,10)(30,5)
            \qbezier(10,15)(10,15)(30,15)
        \end{picture}
    \end{center}
    \caption{Counterexample to $\dmm(G)+\dmm(\Gc) \leq \ord G-1$.}
    \label{counterDmm}
\end{figure}

Another notable result is the relation between degeneracy and the size of a graph.

\begin{prop}[Lick and White \cite{LickWhite}] \label{PropLickWhite}
    Let $G$ be a graph of order $n$ and degeneracy $k$. Then $G$ has at most $kn - \frac{k(k+1)}{2}$ edges.
\end{prop}

Given $n > 0$ and nonnegative integers $h, k < n$, we say that $(h,k)$ is a {\em degeneracy pair of order $n$} if there exists a graph $G$ of order $n$ such that $h = l(G)$ and $k = l(G^c)$. Proposition~\ref{lGprop} implies 

\begin{equation} \label{degpair1}
    h + k \leq n - 1,
\end{equation}
while Proposition~\ref{PropLickWhite} yields

\begin{equation} \label{degpair2}
    hn - \frac{h(h+1)}{2} + kn - \frac{k(k+1)}{2} \geq \frac{n(n-1)}{2}.
\end{equation}

More in general, given nonnegative integers $h, k < n$ we will say that $(h,k)$ is a {\em covering pair of order $n$} if (\ref{degpair1}) and (\ref{degpair2}) are satisfied. Of course, any degeneracy pair is a covering pair. If $(h,k)$ with $h + k \leq n - 2$ is a covering pair of order $n$, then so are $(h+1,k)$ and $(h,k+1)$. On the other hand, (\ref{degpair2}) may fail when $h$ or $k$ are reduced. We then say that a covering pair $(h,k)$ is {\em left-minimal (resp.\ right-minimal)} if $(h-1,k)$ (resp.\ $(h,k-1)$) is not a covering pair. Note that a pair may be right-minimal without being left-minimal or conversely, as it is for the pair $(4,2)$ when $n = 9$.

We also say that a nonnegative integer $r<n$ is a {\em covering (resp.\ degeneracy) sum of order $n$} if there exists a covering (resp.\ degeneracy) pair $(h,k)$ such that $h+k = r$. It can be easily seen that if $(h,k)$ is a covering pair of order $n$, so is $\left(\left\lceil \frac{h+k}{2}\right\rceil, \left\lfloor \frac{h+k}{2}\right\rfloor\right)$. Therefore,

\begin{prop} \label{degsum}
    A nonnegative integer $r$ is a covering sum of order $n$ if and only if $\left(\left\lceil \frac{r}{2}\right\rceil, \left\lfloor \frac{r}{2}\right\rfloor\right)$ is a covering pair of order $n$.
\end{prop}

To determine whether a given covering sum $r$ of order $n$ can be regarded as the degeneracy sum for a suitable graph $G$ of order $n$, we introduce Algorithm \ref{lGalgorithm}. If $G$ is a graph and $S$ is a subset of vertices, then $G[S]$ denotes the subgraph of $G$ induced by the vertices in $S$.

We prove the correctness of Algorithm~\ref{lGalgorithm} in Theorem~\ref{correctness}.
In this regard, if, for a given graph $G$ of order $n$, we consider an ordered sequence of its vertices $\sigma = (v_1,v_2,\ldots,v_n)$, we say that $\sigma$ is a {\em building sequence of degree $h$} if $\max_i\{d_i\}=h$, where, for each $i$, $d_i$ is the degree of $v_i$ in $G[\{v_1,\ldots,v_i\}]$. In fact, the smallest $h$ for which a building sequence of degree $h$ exists coincides with $l(G)$.

\begin{algorithm}[H]
    \caption{Given integers $n > h \geq 0$, this algorithm determines the adjacency matrix $A$ of a graph $G$ with $\ord G=n$, $l(G)=h$, such that $(l(G),l(G^c))$ is a right-minimal covering pair of order $n$.
    }\label{lGalgorithm}
    \begin{algorithmic}[1]
        \Require $n$ positive integer
        \Require $h$ nonnegative integer smaller than $n$
        \Ensure $A$ adjacency matrix of order $n$
        \State $\mathrm{maxlGc}(n,h)$
        \State $A = A(a_{i,j}) = 0_{n \times n}$
        \For{$j$ = $1$ to $n$}
            \State $L(j) = 0$ \qquad // $L(j)$ tracks the degree of vertex $j$ in $G^c[j\colon\! n,j\colon\! n]$
        \EndFor
        \For{$i$ = $2$ to $n$}
            \If{$i \leq h+1$}
                \State $S = \{ 1,2,\ldots,i-1\}$\label{S1} \quad // vertices to be connected to $i$ in  $G[1\colon\! i,1\colon\! i]$
            \Else{}
                \State $S = \varnothing$
                \For{$s$ = $1$ to $h$}\label{S2}
                    \State $j_s$ = index of the $s^{\mathrm{th}}$ largest entry of vector $L$; in case of a tie, choose the smallest index
                    \State $S = S \cup \{j_s\}$\label{S2b}
                \EndFor
            \EndIf
            \For{$j$ = $1$ to $i-1$}
                \If{$j \in S$}\label{if_jS}
                    \State$a_{i,j} = 1$; $a_{j,i} = 1$\label{S3}\quad  // vertex $j$ is adjacent to $i$ in $G[1\colon\! i,1\colon\!i]$
                \Else                   
                    \State $L(j) = L(j) + 1$\quad \label{updateL} // vertex $i$ is adjacent to $j$ in $G^c[j\colon\!n,j\colon\!n]$
                \EndIf
            \EndFor
        \EndFor
    \end{algorithmic}
\end{algorithm}

\begin{thm} \label{correctness}
    For each pair of integers $0 \leq h < n$, Algorithm~\ref{lGalgorithm} determines the adjacency matrix $A$ of a graph $G$ of order $n$ and degeneracy $h$ such that $(l(G),l(G^c))$ is a right-minimal covering pair of order $n$. In addition, building sequences of order $l(G)$ and $l(G^c)$ are $\sigma = (1,2,\ldots,n-1,n)$ and $\tau = (n,n-1,\ldots,2,1)$, respectively.
\end{thm}

\begin{proof}
    Lines~\ref{S1} and \ref{S3} ensure that, for each $i=1,\ldots,h+1$, all the entries to the left of $a_{i,i}$ are $1$s, while for $i=h+2,\ldots,n$, lines~\ref{S2}-\ref{S2b} ensure that precisely $h$ entries to the left of $a_{i,i}$ are $1$s. This implies that $\sigma = (1,2,\ldots,n-1,n)$ is a building sequence of degree $h$ (and thus $l(G) \leq h$), and that the size of $G$ is
    \begin{equation} \label{sizeG}
        m(G)=\frac{h(h+1)}{2}+h(n-h-1) = hn - \frac{h(h+1)}{2},
    \end{equation}
    so that $l(G)=h$ by Proposition~\ref{PropLickWhite}. 
    
    To determine the degeneracy of $G^c$, we observe that line~\ref{updateL} increases $L(j)$ by $1$ exactly when the entry $a_{i,j}$ is not set to $1$ on line~\ref{S3}. Thus, $j < i$ implies that, for each $j=1,\ldots,n$, $L(j)$ counts the number of entries equal to $0$ below $a_{j,j}$. To better compare different iterations of the vector $L$, let $L_i$ denote the vector $L$ after the update on line~\ref{updateL}, for $i = 2,\ldots,n$. Also, let
    \begin{align*}
        t_i = \max_j L_i(j); \qquad \sigma_i = \sum_j L_i(j).
    \end{align*}

    If we define $\tau = (n,n-1,\ldots,2,1)$, the degree of $\tau$ as a building sequence for $G^c$ is $t_n$, so that 
    \begin{equation} \label{degGc}
        l(G^c) \leq t_n
    \end{equation}
    
    When $i \leq h+1$, lines~\ref{S1} and \ref{if_jS}-\ref{updateL} ensure that $L_i=0$, and thus $t_i=0$ and $\sigma_i=0$. On the other hand, when $i \geq h+2$, since $|S|=h$, line~\ref{updateL} is executed $i - 1 - h>0$ times, so that $L_i \neq 0$ and $t_i>0$. To better track the count of the components of $L_i$ that are equal to a specific value $r$, 
    we introduce counting functions $\psi_i$ as follows:
    \begin{equation*}
        \begin{array}{cccc}\psi_i: & \{0,1,\ldots,t_i\} & \rightarrow & \mathbb N\\
        & r & \mapsto & \left|\{ j \colon L_i(j)=r \}\right|
        \end{array}\, .
    \end{equation*}

    For each $i$ we also set $p_i=\psi_i(t_i-1)$ and $q_i=\psi_i(t_i)$. Since $S$ consists of the indices relative to the largest $h$ components of $L$ and such components are not updated in line~\ref{updateL}, one can easily conclude that 
    \begin{equation} \label{titi-1}
     t_i = \left\{ \begin{array}{ll}
     t_{i-1} & \mbox{when $q_{i-1} \leq h$}\\
     t_{i-1}+1 & \mbox{when $q_{i-1}>h$.}
     \end{array} \right.
    \end{equation}
    
    \vspace{1em}
    \noindent{\bf Fact 1.} For each $i \geq h+2$, if $t_i=1$, then $1 \leq q_i \leq i-1$, while if $t_i \geq 2$ the count can be summarized as follows:

    \begin{equation}
        \begin{array}{c|ccccccc} \label{psi}
            r=L_i(j)&0&1&2&\ldots&t_i-2&t_i-1&t_i\\ \hline
            \psi_i(r)&n-i+1&1&1&\ldots&1&p_i&q_i
        \end{array}
    \end{equation}
    with $p_i \geq 1$, $q_i \geq 1$, and $p_i+q_i \geq h$. Note that in (\ref{psi}), the option $\psi_i(1)=1$ should be included only when $t_i \geq 3$, and $\psi_i(2)=\cdots= \psi_i(t_i-2)=1$ only when $t_i \geq 4$. 
    
    To prove Fact~1, we observe that, when $t_i=1$,  line~\ref{updateL} possibly updates only the first $i-1$ components of $L$, so that $q_i \leq i-1$. If $t_i \geq 2$, we proceed by induction on $i$.

    \vspace{1em}
    \noindent Basis Step: $t_i = 2$ and $t_{i-1}=1$. By (\ref{titi-1}) we have $q_{i-1}>h$. Thus, all but $h$ of the components $L_{i-1}(j)=1$ and all the components $L_{i-1}(j)=0$ with $j<i$ are increased by $1$ on line~\ref{updateL}. This yields $q_i = q_{i-1} - h$, so that $1 \leq q_i \leq i-2 - h$, and $p_i = i-1-q_i$, so that $p_i \geq 1+h \geq 1$ and $p_i+q_i = i-1 \geq h$. Hence this count can be summarized as

    \begin{equation*}
        \begin{array}{c|ccc}
            r=L_i(j)&0&1&2\\ \hline
            \psi_i(r)&n-i+1&p_i&q_i
        \end{array}\, .
    \end{equation*}

    \vspace{1em}
    \noindent Inductive Step Case 1: $t_i = t_{i-1} \geq 2$. We can then assume that the table for $\psi_{i-1}$ is 
    \begin{equation}
        \begin{array}{c|ccccccc} \label{psi1}
            r=L_{i-1}(j)&0&1&2&\ldots&t_{i-1}-2&t_{i-1}-1&t_{i-1}\\ \hline
            \psi_{i-1}(r)&n-i+2&1&1&\ldots&1&p_{i-1}&q_{i-1}
        \end{array}
    \end{equation}
    By (\ref{titi-1}) we also have $q_{i-1} \leq h$. The inductive hypothesis also ensures $p_{i-1}+q_{i-1} \geq h$, so that none of the indices for which $L_{i-1}(j) \leq t_{i-1}-2$ is placed in $S$. Therefore, line~\ref{updateL} will increase by $1$ all the components of $L$ such that $1 \leq L_{i-1}(j) \leq t_{i-1}-2$. After the update, it can be easily determined that $q_i = p_{i-1}+2q_{i-1}-h \geq 1$, $p_i=h-q_{i-1}+1 \geq 1$, and $p_i+q_i = p_{i-1}+q_{i-1}+1 \geq h$. Also, when $t_i \geq 3$, $\psi_i(1)=1$ comes from $L_i(i-1)$ which was line~\ref{updateL} updated from $0$ to $1$, while, when $t_i \geq 4$, for $2 \leq r \leq t_i-2$, $\psi_i(r) = \psi_{i-1}(r-1)=1$. Finally, $\psi_i(0)=\psi_{i-1}-1=n-i$, since now $L_i(i-1) \neq 0$.
    
    \vspace{1em}
    \noindent Inductive Step Case 2: $t_i = t_{i-1} + 1 \geq 3$. Again, we can assume the table for $\psi_{i-1}$ to be as in (\ref{psi1}), while (\ref{titi-1}) now yields $q_{i-1}>h$. With reasoning similar to that as presented in the Basis Step, we conclude that $q_i = q_{i-1}-h \geq 1$, and $p_i = p_{i-1}+h \geq 1$, from which $p_i+q_i \geq h$. As in the Inductive Step Case I, when $t_i \geq 3$, then $\psi_i(1)=1$, and when $t_i \geq 4$ and $2 \leq r \leq t_i - 2$ then $\psi_i(r)=\psi_{i-1}(r-1)$, as well as $\psi_i(0)=n-i$, so the proof of Fact~1 is complete.

    Rewriting (\ref{psi}) for $i=n$ we have

    \begin{equation}
        \begin{array}{c|ccccccc} \label{psin}
            r=L_n(j)&0&1&2&\ldots&t_n-2&t_n-1&t_n\\ \hline
            \psi_n(r)&1&1&1&\ldots&1&p_n&q_n
        \end{array}
    \end{equation}

    From (\ref{psin}) we can compute
    \begin{equation*}
        \sigma_n = 0 + 1 + 2 + \cdots + (t_n - 2) + p_n(t_n -1) + q_n(t_n).
    \end{equation*}
    This yields lower and upper bounds on $\sigma_n$ in terms of $t_n$. In fact, observing that necessarily $t_n - 1 + p_n + q_n = n$, the maximum value of $\sigma_n$ is achieved when $p_n = 1$ and $q_n = n - t_n$, so that

    \begin{equation} \label{SigmaUpper}
        \sigma_n \leq \frac{(t_n-1)t_n}{2} + (n - t_n)t_n\, , 
    \end{equation}
    while the minimum value is attained when $q_n = 1$ and $p_n = n - t_n$, so that
    \begin{equation} \label{SigmaLower}
        \sigma_n \geq \frac{(t_n-2)(t_n-1)}{2} + (n - t_n)(t_n - 1) + t_n\, .
    \end{equation}
    
    Simple algebraic computations combine (\ref{SigmaUpper}) and (\ref{SigmaLower}) giving
    \begin{equation} \label{SigmaLowerUpper}
        (t_n - 1)n - \frac{(t_n-1)t_n}{2} + 1 \leq \sigma_n \leq t_n n - \frac{t_n(t_n+1)}{2} \,.
    \end{equation}

    We also observe that $\sigma_n$ coincides with the size of $G^c$, so that, recalling (\ref{sizeG}) we have
    \begin{equation} \label{hSigma}
        hn - \frac{h(h+1)}{2} + \sigma_n = \frac{n(n-1)}{2}
    \end{equation}

    Finally, note that the inequality to the left in (\ref{SigmaLowerUpper}) together with (\ref{hSigma}) implies that $(h,t_n-1)$ does not satisfy (\ref{degpair2}) and thus is not a covering pair. On the other hand, $(l(G),l(G^c))$ is necessarily a covering pair. Recalling that $l(G)=h$, by (\ref{degGc}) we conclude $l(G^c)=t_n$, and therefore $(l(G),l(G^c)=(h,t_n)$ is a right-minimal covering pair.
\end{proof}

\begin{cor} \label{deg=cov}
    Let $0 \leq r < n$. Then $r$ is a degeneracy sum of order $n$ if and only if $r$ is a covering sum of order $n$.
\end{cor}

\begin{proof}
    We already observed that a degeneracy sum of order $n$ is a covering sum of the same order. Conversely, if $r$ is a covering sum of order $n$, among all covering pairs $(h,k)$ with $h + k = r$ select one with minimum $k$. By applying Algorithm~1 we obtain a graph with $|G|=n$, $l(G)=h$, and $l(G^c)=t_n$, where $(h,t_n)$ is a right-minimal covering pair. The minimality of $t_n$ yields $k \geq t_n$. On the other hand, if $k > t_n$, then $(h+k-t_n,t_n)$ would also be a covering pair with sum $r$, against the minimality of $k$.
    Thus $k = t_n$, so that $l(G) + l(G^c) = h + k = r$.
\end{proof}

\begin{thm} \label{lGbounds}
    Let $0 \leq r < n$. Then $r$ is a covering sum of order $n$ if and only if
    \begin{equation} \label{brace_r}
        r \geq \left\{
        \begin{array}{ll}
        2n - 1 - \sqrt{2n^2-2n+1} & \mbox{if $r$ is even}\\
        2n - 1 - \sqrt{2n^2-2n} & \mbox{if $r$ is odd}
        \end{array}
        \right.
    \end{equation}
\end{thm}

\begin{proof}
    We first consider the case where $r$ is even, that is, $r = 2h$ for some integer $h$. In this case, by making use of Proposition \ref{degsum}, equation (\ref{degpair2}) becomes
    \begin{align*}
        2hn - h(h+1) \geq \frac{n(n-1)}{2}.
    \end{align*}
    which can be rearranged into the quadratic inequality
    \begin{align*}
        2h^2 - 2(2n - 1)h + n^2 - n \leq 0.
    \end{align*}
    which yields 
    \begin{equation} \label{range_even}
    \frac{2n - 1 - \sqrt{2n^2 - 2n + 1}}{2} \leq h \leq \frac{2n - 1 + \sqrt{2n^2 - 2n + 1}}{2}.
    \end{equation}
    The inequality on the right is superseded by $h < \frac{n}{2}$, so that, when $r$ is even, (\ref{range_even}) is equivalent to (\ref{brace_r}).

    When $r$ is odd we just set $r = 2h - 1$. A reasoning similar to the even case yields
    \begin{equation*}
    \frac{2n - \sqrt{2n^2 - 2n}}{2} \leq h \leq \frac{2n + \sqrt{2n^2 - 2n}}{2}.
    \end{equation*}
    from which $r \geq 2n - 1 - \sqrt{2n^2 - 2n}$ follows.
\end{proof}

Theorem~\ref{lGbounds} allows us to identify Nordhaus-Gaddum lower and upper bounds for the parameter $l(G)$ in terms of the number of vertices in $G$.

\begin{thm}\label{thm_Nord-Gadd_for_degen}
    For any graph $G$ of order $n$
    \begin{equation} \label{NG}
        2n - 1 - \sqrt{2n^2 - 2n + 1} \leq l(G) + l(G^c) \leq n -1.
    \end{equation}
    In addition, for any integer $r$ in such a range there exists a graph $G$ of order $n$ with $l(G)+l(G^c)=r$.
\end{thm}

\begin{proof}
    The degeneracy sum $l(G) +l(G^c)$ is a covering sum by Corollary~\ref{deg=cov}, and so it satisfies (\ref{brace_r}). Together with (\ref{lG+lGc}) it implies (\ref{NG}).
    For the second part, for a given integer $r$ within the range determined by (\ref{NG}), if $r$ is even then (\ref{brace_r}) holds, and so $r$ is a covering sum, and thus a degeneracy sum. For $r$ odd the same argument holds except when
    \begin{equation} \label{odd_scenario}
        r = \left\lceil 2n - 1 - \sqrt{2n^2 - 2n + 1}\right\rceil < \left\lceil 2n - 1 - \sqrt{2n^2 - 2n} \right\rceil
    \end{equation}
    which occurs when $\left\lfloor \sqrt{2n^2 - 2n + 1}\right\rfloor > \left\lfloor \sqrt{2n^2 - 2n} \right\rfloor$. Note that, since the arguments in these two radicands differ by $1$, the inequality occurs only when $2n^2 - 2n + 1$ is a perfect square. Thus, $\sqrt{2n^2-2n+1}$ is an odd integer, so that $r$ would be even. In other words, (\ref{odd_scenario}) cannot occur when $r$ is odd, and the proof is complete.
\end{proof}

\begin{ex}
    For $n = 14$, equation (\ref{NG}) yields $l(G)+l(G^c) \geq 8$. By applying Algorithm~\ref{lGalgorithm}, we obtain a graph $G$ (see Figure~\ref{fig:enter-label}) such that $l(G) = l(G^c) = 4$.
\end{ex}
\begin{figure}[H]
    \centering
\begin{tikzpicture}[every node/.style={circle, draw=black, ultra thick, fill=white, minimum size=14pt, inner sep=0.5pt}, scale=1]

% Define node positions
\node (1) at (0.5,0) {1};
\node (2) at (1,1) {2};
\node (3) at (1,-1) {3};
\node (4) at (2,2) {4};
\node (5) at (2,-2) {5};
\node (6) at (3,2.5) {6};
\node (7) at (3,-2.5) {7};
\node (8) at (4,2.5) {8};
\node (9) at (4,-2.5) {9};
\node (10) at (5,2) {10};
\node (11) at (5,-2) {11};
\node (12) at (6,1) {12};
\node (13) at (6,-1) {13};
\node (14) at (6.5,0) {14};

% edges for vertex 1.
\draw[thick] (1) -- (2);
\draw[thick] (1) -- (3);
\draw[thick] (1) to[bend right=10] (4);
\draw[thick] (1) to[bend left=10] (5);
\draw[thick] (1) to[bend right=10] (6);
\draw[thick] (1) to[bend left=10] (7);
\draw[thick] (1) to[bend right=5] (8);
\draw[thick] (1) -- (10);
\draw[thick] (1) -- (12);
\draw[thick] (1) -- (14);

% edges for vertex 2.
\draw[thick] (2) -- (3);
\draw[thick] (2) -- (4);
\draw[thick] (2) -- (5);
\draw[thick] (2) to[bend right=10] (6);
\draw[thick] (2) -- (7);
\draw[thick] (2) -- (9);
\draw[thick] (2) -- (10);
\draw[thick] (2) -- (12);

% edges for vertex 3.
\draw[thick] (3) -- (4);
\draw[thick] (3) -- (5);
\draw[thick] (3) -- (6);
\draw[thick] (3) to[bend left=5] (7);
\draw[thick] (3) -- (9);
\draw[thick] (3) -- (11);
\draw[thick] (3) -- (13);

% edges for vertex 4.
\draw[thick] (4) -- (5);
\draw[thick] (4) -- (6);
\draw[thick] (4) to[bend right=10] (8);
\draw[thick] (4) -- (9);
\draw[thick] (4) -- (11);
\draw[thick] (4) -- (13);

% edges for vertex 5.
\draw[thick] (5) -- (7);
\draw[thick] (5) -- (8);
\draw[thick] (5) to[bend left=4] (9);
\draw[thick] (5) -- (11);
\draw[thick] (5) -- (13);

% edges for vertex 6.
\draw[thick] (6) -- (8);
\draw[thick] (6) to[bend right=10] (10);
\draw[thick] (6) -- (11);
\draw[thick] (6) -- (13);

% edges for vertex 7.
\draw[thick] (7) -- (10);
\draw[thick] (7) -- (12);
\draw[thick] (7) -- (14);

% edges for vertex 8.
\draw[thick] (8) to[bend right=10] (12);
\draw[thick] (8) to[bend right=10] (14);

% edges for vertex 9.
\draw[thick] (9) to[bend left=5] (14);
\end{tikzpicture}
\hspace{0.8cm}
\begin{tikzpicture}[every node/.style={circle, draw=black, ultra thick, fill=white, minimum size=14pt, inner sep=0.5pt}, scale=1]
% Define node positions
\node (1) at (0.5,0) {1};
\node (2) at (1,1) {2};
\node (3) at (1,-1) {3};
\node (4) at (2,2) {4};
\node (5) at (2,-2) {5};
\node (6) at (3,2.5) {6};
\node (7) at (3,-2.5) {7};
\node (8) at (4,2.5) {8};
\node (9) at (4,-2.5) {9};
\node (10) at (5,2) {10};
\node (11) at (5,-2) {11};
\node (12) at (6,1) {12};
\node (13) at (6,-1) {13};
\node (14) at (6.5,0) {14};

% edges for vertex 1.
\draw[thick] (1) -- (9);
\draw[thick] (1) -- (11);
\draw[thick] (1) -- (13);

% edges for vertex 2.
\draw[thick] (2) -- (8);
\draw[thick] (2) -- (11);
\draw[thick] (2) -- (13);
\draw[thick] (2) -- (14);

% edges for vertex 3.
\draw[thick] (3) -- (8);
\draw[thick] (3) -- (10);
\draw[thick] (3) to[bend left=8] (12);
\draw[thick] (3) -- (14);

% edges for vertex 4.
\draw[thick] (4) -- (7);
\draw[thick] (4) -- (10);
\draw[thick] (4) -- (12);
\draw[thick] (4) -- (14);

% edges for vertex 5.
\draw[thick] (5) -- (6);
\draw[thick] (5) to[bend left=10] (10);
\draw[thick] (5) -- (12);
\draw[thick] (5) -- (14);

% edges for vertex 6.
\draw[thick] (6) -- (7);
\draw[thick] (6) to[bend left=5] (9);
\draw[thick] (6) -- (12);
\draw[thick] (6) -- (14);

% edges for vertex 7.
\draw[thick] (7) -- (8);
\draw[thick] (7) -- (9);
\draw[thick] (7) to[bend left=8] (11);
\draw[thick] (7) -- (13);

% edges for vertex 8.
\draw[thick] (8) -- (9);
\draw[thick] (8) -- (10);
\draw[thick] (8) -- (11);
\draw[thick] (8) -- (13);

% edges for vertex 9.
\draw[thick] (9) -- (10);
\draw[thick] (9) -- (11);
\draw[thick] (9) -- (12);
\draw[thick] (9) to[bend left=8] (13);

% edges for vertex 10.
\draw[thick] (10) -- (11);
\draw[thick] (10) -- (12);
\draw[thick] (10) -- (13);
\draw[thick] (10) to[bend right=10] (14);

% edges for vertex 11.
\draw[thick] (11) -- (12);
\draw[thick] (11) -- (13);
\draw[thick] (11) to[bend left=5] (14);

% edges for vertex 12.
\draw[thick] (12) -- (13);
\draw[thick] (12) -- (14);

% edges for vertex 13.
\draw[thick] (13) -- (14);
\end{tikzpicture}
\caption{Complementary graphs of order $14$ with $l(G)=l(G^c)=4$}
    \label{fig:enter-label}
\end{figure}
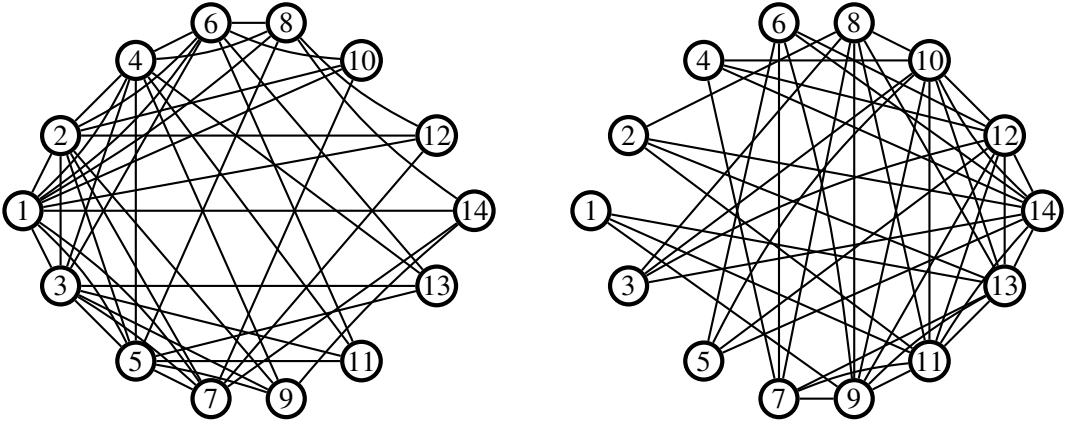

\subsection{Weak graph complement conjecture improvements: variations}\label{subsec_deg_WGCC}

In this subsection, we return to the Weak Graph Complement Conjecture and observe particular updates to the proposed bounds, upon assuming the validity of the $\delta$-conjecture and referring to the results on graph degeneracy (see Theorem~\ref{thm_Nord-Gadd_for_degen}).

We begin by noting that the role of degeneracy in relation to maximum nullity of graphs with strong Arnold property which has been recently studied in \cite{LonDeg}.

\begin{thm} [Mitchel \cite{LonDeg}]
    For any graph $G$
    \begin{equation} \label{LonEq}
        \nu(G) \geq \ord G - 2l(\Gc) - 1.
    \end{equation}
\end{thm}

The next two results resolve Problem 2 and  establish bounds $b_{\nu}$ for the strongest version of the Weak Graph Complement Conjectures \ref{gccnu}, assuming that the $\delta$-conjecture~\ref{delta} holds for all graphs. Note that the first result makes a simple use of Proposition~\ref{lGprop} 

\begin{prop} \label{obs_half}
    If we assume the validity of the $\delta$-conjecture, we conclude that, for any graph $G$,
    \begin{equation} \label{weekHalf}
        \nu(G) + \nu(\Gc) \geq \dfrac{\ord G-1}{2}.
    \end{equation}
    or, equivalently,
    \begin{equation} \label{week3Half}
        \mrn(G)+\mrn(\Gc) \leq \dfrac{3}{2}\ord G+\dfrac{1}{2}.
    \end{equation}
    
\end{prop}

\begin{proof}
    By applying (\ref{deltaEq}) to $\Gc$ we have $\nu(\Gc) \geq l(\Gc)$, which can be added to (\ref{LonEq}) to obtain $\nu(G)+\nu(\Gc) \geq \ord G - l(\Gc) - 1.$ In a similar way we can also obtain $\nu(G)+\nu(\Gc) \geq \ord G - l(G) -1,$ which, by  recalling Proposition~\ref{lGprop}, yields
    \begin{align*}
        \nu(G)+\nu(\Gc) &\geq \max\{\ord G - l(G) -1,\ord G - l(\Gc) -1\}
         = \ord G - \min\{l(G),l(\Gc)\}-1
         \geq  \dfrac{\ord G-1}{2}.
    \end{align*}
    Equation (\ref{week3Half}) easily follows from (\ref{weekHalf}).
\end{proof}

Before we move to the next result, we note that the positive semidefinite version of Conjecture \ref{gccnu} was shown to hold for the special case of bipartite graphs in \cite{mrplus-Siv} with a bound of 3/2. 

However, applying Theorem~\ref{thm_Nord-Gadd_for_degen}, we derive an even tighter bound on the strongest version of the Weak Graph Complement Conjecture \ref{gccnu}, again assuming the $\delta$-conjecture~\ref{delta} holds.

\begin{prop}
    If we assume the validity of the $\delta$-conjecture, we can improve further the claim in Proposition \ref{obs_half} and conclude that, for any graph $G$,
    \begin{equation*}
        \nu(G) + \nu(\Gc) \geq (2-\sqrt{2})\ord G - 1.
    \end{equation*}
    or, equivalently,
    \begin{equation*}
       \mrn(G)+\mrn(\Gc) \leq \sqrt{2}\ord G + 1.
    \end{equation*}
\end{prop}

\begin{proof}
    It suffices to apply Theorem~\ref{thm_Nord-Gadd_for_degen} and observe that 
    $
        (2-\sqrt{2})n -1 \leq 2n - 1 - \sqrt{2n^2-2n+1}.
    $
\end{proof}

\section{Bounds in terms of other well-known graph parameters}\label{sec_other_lower_bounds}

In this section, we consider bounds, some verified and some suspected, on the graph parameter $\nu(G)$. We denote the chromatic number of a graph $G$ by $\chi(G)$. It is well known that every graph $G$ contains a subgraph whose minimum degree is at least $\chi(G)-1$, that is, $l(G)\geq \chi(G)-1$. The Hadwiger number, $\eta(G)$, of a graph $G$ is the largest integer $t$ such that the complete graph $K_t$ is a minor of $G$. Consequently, it is clear that $\lceil \kappa \rceil (G)\geq \eta(G)-1$, and hence by Theorem \ref{thm_LSS_vertex_connectivity} $\nu(G)\geq \eta(G)-1$. One of the most famous conjectures in graph theory is `Hadwiger's conjecture' which asserts that $\eta(G)\geq \chi(G)$. Therefore, if either $\delta$-Conjecture \ref{delta} or Hadwiger's Conjecture is true, then $\nu(G)\geq \chi(G)-1$. However, to the best of these authors' knowledge, the validity of this inequality in general remains an open problem. We thus state it here as a conjecture:
\begin{conj}\label{conj_nu_chromatic}
For any graph $G$, we have $\nu(G)\geq \chi(G)-1$.    
\end{conj}

Using the following recent result by Nguyen \cite{Nguyen2024} and Theorem \ref{thm_LSS_vertex_connectivity}, one can observe that, $\nu(G)$ is near to one-third of the value predicted by Conjecture \ref{conj_nu_chromatic}.

\begin{thm}[Nguyen \cite{Nguyen2024}]
    For every integer $k\geq 1$, every graph $G$ with $\chi(G) \geq \frac{49}{16}\cdot k$ contains a $(k+1)$-connected subgraph with more than $\chi(G)-k$ vertices and chromatic number at least $\chi(G)-2k+2$.
\end{thm}

\begin{cor}
Let $G$ be a graph with chromatic number $\chi(G)$. Then $\nu(G)> \frac{16}{49}\chi(G)$.    \end{cor}

We denote the independence number of a graph $G$ by $\alpha(G)$. It is well known that $\chi(G)\geq \frac{n}{\alpha(G)}$. Therefore, if Conjecture \ref{conj_nu_chromatic} holds true, then $\nu(G)\geq \frac{n}{\alpha(G)}-1$. Again, to the best of the authors' knowledge, this is not known to hold in general, so we state it here as another interesting but weaker conjecture:
  
\begin{conj}\label{conj_nu_independence}
 For any graph $G$, we have $\nu(G)\geq \frac{n}{\alpha(G)}-1$.      
\end{conj}

A weakening of Hadwiger's conjecture is that $\eta(G)\geq \frac{n}{\alpha(G)}$. The following result by Balogh and Kostochka \cite{Balogh2011} is the best current bound to date for this weak version of Hadwiger's conjecture. Using this result and Theorem \ref{thm_LSS_vertex_connectivity}, one can observe that, $\nu(G)$ is more than one-half of the value predicted by Conjecture \ref{conj_nu_independence}. 

\begin{thm}[Balogh and Kostochka \cite{Balogh2011}]
    Let $G$ be a graph of order $n$. Then $\eta(G) \geq \frac{n}{(2-c)\alpha(G)}$ where $c = (80-\sqrt{5392})/126 > 1/19.2$.
\end{thm}

\begin{cor}
Let $G$ be a graph with independence number $\alpha(G)$. Then $\nu(G)\geq \frac{n}{(2-c)\alpha(G)}-1$ where $c = (80-\sqrt{5392})/126 > 1/19.2$.   
\end{cor}

We summarize these lower bounds (solid lines) and suspected relationships (dashed or dotted lines) concerning $\nu(G)$ below in Figure \ref{fig2}.
\begin{figure}[H]
    \centering
    \begin{tikzpicture}[scale=1]
\node[] (A) at (0.35,-0.5) {{\small $\frac{n}{\alpha(G)}-1$}}; 
\node[] (B) at (1.75,0.5) {{\small $\chi(G)-1$}};
\node[] (E) at (3.35,2.5) {{\small$\lceil \delta \rceil (G) $}};
\node[] (F) at (3,4) {{\small $\nu(G)$}};
\node[] (G) at (5,2) {{\small$\lceil \kappa \rceil(G) $}};
\node[] (H) at (5,1) {{\small$\eta(G)-1$}};

\draw[->,very thick] (G) -- (F);  
\draw[->,very thick] (H) -- (G);
\draw[->,very thick] (G) -- (E);
\draw[->,very thick] (B) -- (E);
\draw[->,very thick] (A) -- (B);
\draw[dotted,very thick] (A) -- (F);
\draw[dotted,very thick] (B) -- (F);
\draw[dash dot,very thick] (A) -- (H);
\draw[dash dot,very thick] (B) -- (H);
\draw[dashed,very thick] (E) -- (F);
\end{tikzpicture}
    \caption{Known and suspected lower bounds on $\nu(G)$}
    \label{fig2}
\end{figure}
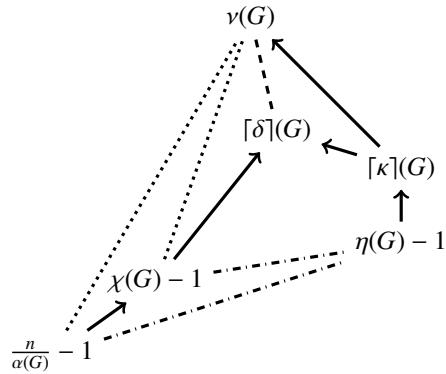

\section*{Acknowledgments}
The authors' collaboration began as part of the established ``Inverse Eigenvalue Problems for Graphs and Zero Forcing” Research Community sponsored by the American Institute of Mathematics (AIM). We thank AIM for their support, and we thank the organizers and participants for contributing to this stimulating research experience.

S.M.\ Fallat was supported in part by an NSERC Discovery Research Grant, Application No.: RGPIN-2019-03934. The work of the PIMS Postdoctoral Fellow H.\ Gupta leading to this publication was supported in part by the Pacific Institute for the Mathematical Sciences.

%%%%%%%%%%%%%%%%%%%%%%%%%%%%%%%%%%%%%%%%%

\end{document}